\documentclass[]{amsart}

\usepackage{amsfonts}
\usepackage{amscd}
\usepackage{epsfig}
\usepackage{epic,eepic}
\usepackage{graphicx}
\usepackage[all]{xy}
\usepackage{xypic}
\usepackage{psfrag}
\usepackage{amsmath}
\usepackage{amsthm}
\usepackage{setspace}
\usepackage{url}
\usepackage{here}
\theoremstyle{plain}
\newtheorem{theorem}{Theorem}[section]
\newtheorem{lemma}[theorem]{Lemma}
\newtheorem{proposition}[theorem]{Proposition}
\newtheorem{corollary}[theorem]{Corollary}

\newtheorem{question}[theorem]{Question}
\theoremstyle{definition}
\newtheorem{remark}[theorem]{Remark}
\newtheorem{definition}[theorem]{Definition}
\newtheorem{example}[theorem]{Example}

\title{Simply Branched Covers of An Elliptic Curve and the Moduli Space of Curves}
\author{Dawei Chen} 
\date{}

\begin{document}
\bibliographystyle{plain}
\maketitle

\newcommand{\Mg}{$\overline{\mathcal M}_{g}\ $}

\begin{abstract}
Consider genus $g$ curves that admit degree $d$ covers to an elliptic
curve simply branched at $2g-2$ points. Vary a branch point and 
the locus of such covers forms a one-parameter family $W$. We investigate the geometry of $W$
by using admissible covers to study its slope, genus and components. The results can also be applied to study slopes of effective divisors on the moduli space of genus $g$ curves. 
\end{abstract}

\tableofcontents

\section{Introduction}
This paper is a sequel of \cite{C}. One of our motivations is to study the birational geometry of the moduli space of 
genus $g$ curves \Mg by using certain one-parameter families of covers of an elliptic curve. 

Recall that Pic$(\overline{\mathcal M}_{g})\otimes \mathbb Q$
is generated by the Hodge class $\lambda$ and boundary classes
$\delta_{i}, i = 0,1,\ldots, [\frac{g}{2}].$ Let $\delta$ denote the total boundary. From the perspective of 
birational geometry, it is natural to ask what linear combinations of $\lambda$ and $\delta$ are ample, nef or effective. 
It is well-known that $\lambda$ itself is nef but not ample. Moreover, Cornalba and Harris established \cite{CH} the following result. 

\begin{proposition}  
For any positive integers $a$ and $b$, the divisor class $a\lambda - b\delta$ is ample on \Mg if and only if $a > 11b$. 
\end{proposition}

People have associated a number called \emph{slope} to a divisor class as follows. 

\begin{definition}
Let $D = a\lambda - \sum_{i=0}^{[g/2]} b_{i}\delta_{i}$ be a divisor on $\overline{\mathcal M}_{g}$, $a, b_{i} > 0$. The slope of $D$ is given by $s(D) = \frac{a}{min \{b_{i}\}}.$
\end{definition} 

In terms of this notation, the above proposition says that $D = a\lambda - b\delta$ is ample if and only if $s(D) > 11$. However, if we want to subtract more boundary from $D$ and maintain $D$ effective, the question below is still open. 

\begin{question}
What is the lower bound of the slope $s(D)$ for all effective divisors $D$ on $\overline{\mathcal M}_{g}$?
\end{question}

The study about slopes of effective divisors on \Mg is interesting in a number of reasons. Here we will briefly mention some of them. 

Probably one of the most important facts is that the existence of an effective divisor with slope smaller than $\frac{13}{2}$ would imply that \Mg is of general type. 
The upshot is that by using such a divisor, we can write the canonical divisor $K_{\overline{\mathcal M}_{g}}$ as an ample divisor plus an effective divisor, 
since $K_{\overline{\mathcal M}_{g}}$ has slope exactly equal to $\frac{13}{2}$. Then $K_{\overline{\mathcal M}_{g}}$ is big, which tells that \Mg is of general type. 
Along these lines, Harris, Mumford and Eisenbud have constructed certain Brill-Noether divisors with small slopes. Eventually, they are able to show the following, cf. 
\cite{HMu}, \cite{H} and \cite{EH}. 

\begin{proposition}
\Mg is of general type when $g$ is greater than 23.  
\end{proposition}

More recently, Farkas, Popa and Khosla found a series of effective divisors, despite that most of them are virtual, with slopes even lower than the Brill-Noether bound $6+\frac{12}{g+1}$, cf. \cite{FP}, \cite{Kh}, 
\cite{F1}, \cite{F2} and \cite{F3}. In particular, Farkas also announces that \Mg is of general type for $g = 22, 23$. 

Surprisingly, all the known effective divisors on \Mg have slope bigger than 6, which makes one guess that there would be a uniform lower bound, independent of $g$, 
for slopes of effective divisors. Let $s_{g}$ denote the \emph{lower bound} for slopes of all effective divisors on $\overline{\mathcal M}_{g}$. 
In \cite{HM1}, \cite{CR}, and \cite{T}, Harris and Morrison, Chang and Ran, as well as Tan among others studied $s_{g}$ for small genera. One of their main results can be summarized as follows.  

\begin{proposition}
$s_{g}$ is equal to the Brill-Noether bound $6 + \frac{12}{g+1}$ for $2 \leq g \leq 9$ and $g = 11$. 
\end{proposition}

The case $g = 10$ is an exception, since $s_{10} = 7$ is strictly less than the Brill-Noether bound. A unique effective divisor with slope 7 has been found 
by Farkas and Popa, cf. \cite{FP}.  

For large $g$, the exact value of $s_{g}$ is completely unknown. Still in \cite{HM1}, Harris and Morrison obtained an asymptotic lower bound for $s_{g}$. 
Unfortunately, the bound is not sharp and seems to approach 0 when $g$ goes to infinity. Nevertheless, their idea is simple and can be generalized in various
directions. 

In order to bound the slope of an effective divisor, we would like to construct moving curves on \Mg and compute their slopes. 

\begin{definition}
A \emph{moving curve} $B$ is an irreducible one-dimensional family of stable genus $g$ curves whose deformations cover an open set of $\overline{\mathcal M}_{g}$. Its \emph{slope} is given by
$s(B) = \frac{B\ldotp \delta}{B\ldotp \lambda}.$
\end{definition}

Sometimes people also use $\frac{\kappa}{\lambda}$ as the slope of the fibration over $B$. Due to the relation
$\lambda = \frac{1}{12}(\kappa + \delta), $ the two definitions determine each other. Our definition above is more consistent with 
the slope of a divisor. 

Since an effective divisor $D$ cannot contain all the deformations of $B$, their intersection $D\ldotp B > 0$, which immediately implies $s(D) > s(B)$.  

Harris and Morrison considered simply branched genus $g$ covers of $\mathbb P^{1}$ with large degree. When the branch points vary along a one-dimensional base, those covering curves form a moving curve $Z$ in $\overline{\mathcal M}_{g}.$ The calculation for the slope $s(Z)$ boils down to a rather complicated enumerative problem, for which they could not even provide a closed formula! However, an experimental study based on computer check seems to yield $s(Z) \sim O(\frac{1}{g})$. 

One way to generalize this idea is to consider one-parameter families of curves in higher dimensional projective spaces. In \cite{CHS}, Coskun, Harris and Starr 
studied canonical curves and obtained sharp bound $s_{g}$ for genus up to 6. The obstruction along this direction is that the geometry of canonical curves is hard to describe in general
when the genus is large. 
 
In \cite{F}, Fedorchuk considered special one-parameter families of curves in the Severi variety and got a recursive formula for their slopes. 
The formula provides some new lower bounds for $s_{g}$ when $g \leq 20$, but the asymptotic estimate of these bounds is unclear. 

Instead of covers to $\mathbb P^{1}$, another try is to consider genus $g$ covers of an elliptic curve that are branched only at one point with a fixed ramification type. When the moduli of the elliptic curve varies, those covering curves form a one-dimensional family $Y$. The geometry of $Y$ has been studied intensively, cf. \cite{C}. In particular, we come up with
a slope formula of $Y$. Unfortunately the difficulty is similar as we encountered in \cite{HM1}. The slope formula involves some enumerative numbers about covers of elliptic curves, which we have not been able to analyze for large $g$ to obtain a numerical bound. 

Recently, Pandharipande kindly explained to the author a different approach, by computing certain intersection numbers of cotangent line bundles with $\lambda$ and $\delta$ on the moduli space of curves with marked points. A lower bound $s_{g} \geq \frac{60}{g+4}$ has been established \cite{P}. Note that this bound tends to 0 as well for large $g$. 

We also notice that the slope of a one-parameter family of stable genus $g$ curves is interesting for its own sake. It usually implies some geometric properties this family possesses. 
For instance, Cornalba and Harris verified the following result, cf. \cite{CH}. 

\begin{proposition}
Any one-parameter family of stable genus $g$ curves, not all singular, must have slope less than or equal to $8 + \frac{4}{g}$. The equality holds only if the family is a hyperelliptic fiberation. 
\end{proposition}

Stankova further studied families of trigonal curves \cite{S} and obtained a bound for their slopes. Her results can be summarized as follows.  

\begin{proposition}
For a trigonal family $X$ over a one-dimensional base $B$, its slope is less than or equal to $\frac{36(g+1)}{5g +1}$, with equality occurring if all fibers are irreducible, $X$ is a triple cover of a ruled surface over $B$, and a certain divisor class on $X$ is numerically zero.
\end{proposition}

Inspired by all the above, in this paper we continue to study the geometry of certain one-parameter families of covers of elliptic curves. 

Fix an elliptic curve $E$ and $2g-2$ points on it.
Consider all genus $g$ covers of $E$ simply branched at these points. When the branch points move along a one-dimensional base, those covers vary and form a one-parameter family 
$W_{d,g}$. Sometimes we will skip the indices $d,g$ if there is no confusion. In the first place it may not be clear what happens when two branch points meet. One might guess that a degenerate singular cover would arise in this process. Luckily there is 
an explicit way to fill in a special cover using so-called \emph{admissible covers}, which was defined in \cite{HMu} by Harris and Mumford. An excellent introduction to admissible covers can also be found in the book \cite[3G]{HM2}. 

Let us formally set up the question. Take the product $S = E\times X$ over the base $X$ a smooth curve, along with $2g-2$ sections $\Gamma_{1}, \ldots, \Gamma_{2g-2}$. We require that no three sections can meet at a common point, and that if two sections meet, they must intersect transversely. Moreover, a fiber $E$ can possess at most one such intersection point. Namely, 
if $\sum_{i < j} \Gamma_{i}\ldotp \Gamma_{j} = k$, then there are $k$ distinct points $b_{1}, \ldots, b_{k}$ on $X$ such that the fiber over each $b_{i}$ has an intersection point of two 
sections. Denote the complement of $\{b_{1}, \ldots, b_{k}\}$ in $X$ by $X^{0}$.  $X$ admits a map to 
$\overline{\mathcal M}_{1, 2g-2}$. Its image has constant $j$-invariant of $E$. Let $\overline{\mathcal H}_{d, g}$ denote the Hurwitz space parameterizing degree $d$ genus $g$ admissible covers
of elliptic curves simply branched at the $2g-2$ marked points, cf. \cite[3G]{HM2}. Then we define the one-parameter family $W$ by the following fiber product: 

$$\xymatrix{
W \ar[r] \ar[d]  &  \overline{\mathcal H}_{d, g}\ar[d] \ar[r] & \overline{\mathcal M}_{g}\\
X \ar[r]     & \overline{\mathcal M}_{1, 2g-2}}$$ 

Note that $\overline{\mathcal H}_{d, g}$ admits a natural map to $\overline{\mathcal M}_{g}$, via which $W$ also maps to $\overline{\mathcal M}_{g}$. We want to study the geometry of $W$, including 
its slope, genus and irreducible components. 

One of our main tasks is to figure out the slope $s(W) = \frac{W\ldotp \delta}{W\ldotp\lambda}$. 
The family $W$ can be viewed as the opposite of the family $Y$ considered in \cite{C}. Generally speaking, once we fix the degree and genus of a cover as well as the ramification type over $n$ branch points on an elliptic curve, there is always a Hurwitz space of admissible covers over $\overline{\mathcal M}_{1, n}$. Any effective curve class in $\overline{\mathcal M}_{1, n}$ pulls back to a one-parameter family of covers in the Hurwitz space that maps to \Mg. Ideally we should compute the slopes of all such families. But in practice the calculation could be quite messy. Hence, we only choose some one-parameter families to study. In \cite{C}, we considered the most degenerate case when there is only one branch point, that is, the $n$ branch points approach all together. Here the family $W$, instead, is the most general case, since the covers are simply branched at $2g-2$ distinct points and any other families can be viewed as its special cases when some of the branch points tend to each other. 

The strategy to compute the slope of $W$ is similar to the methods used in \cite{HM1} and \cite{C}. We first associate to each cover some combinatorial data by using the symmetric group
$S_{d}$, and then study the degenerate admissible covers based on the combinatorial data. 

Take a standard torus $E$ with $2g-2$ distinct marked points $p_{1}, \ldots , p_{2g-2}$. Pick a base point $b$ on $E$ and a group of generators $\alpha, \beta, \gamma_{1}, \ldots, 
\gamma_{2g-2}$ of $\pi_{1}(E_{b}, p_{1}, \ldots, p_{2g-2})$, where $\alpha, \beta$ are a standard basis for $\pi_{1}(E_{b})$ and $\gamma_{i}$ is a closed path going around $p_{i}$ once. 
After choosing suitable directions for those paths, we have the relation  
$ \beta^{-1}\alpha^{-1}\beta\alpha \sim \gamma_{1} \cdots \gamma_{2g-2}, $
which can be seen from Figure \ref{E}. 
\begin{figure}[H]
    \centering
    \psfrag{E}{$E$}
    \psfrag{p1}{$p_{1}$}
    \psfrag{pi}{$p_{i}$}
    \psfrag{pg}{$p_{2g-2}$}
    \psfrag{b}{$b$}
    \psfrag{r1}{$\gamma_{1}$}
    \psfrag{ri}{$\gamma_{i}$}
    \psfrag{rg}{$\gamma_{2g-2}$}
    \psfrag{a}{$\alpha$}
    \psfrag{be}{$\beta$}
    \includegraphics[scale=0.3]{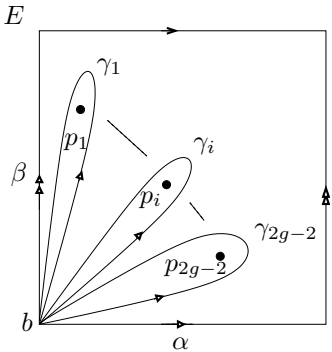}
    \caption{\label{E} $E$ with marked points and paths}
\end{figure}
A degree $d$ genus $g$ cover of $E$ simply branched at $p_{1}, \ldots, p_{2g-2}$ corresponds to a homomorphism $\pi_{1}(E_{b}, p_{1}, \ldots, p_{2g-2}) \rightarrow S_{d}$, 
such that the images of all $\gamma_{i}$'s are simple transpositions. We abuse the notation and still use $\alpha, \beta, \gamma_{1}, \ldots, \gamma_{2g-2}$ to denote their monodromy 
images in $S_{d}$. Define a set $Cov$ as follows. 

\begin{definition}
\label{Cov} Let the set $Cov$ be
$$Cov:= \{ (\alpha, \beta, \gamma_{1}, \ldots, \gamma_{2g-2})\in S_{d} \times \cdots \times S_{d} \ |\ \beta^{-1}\alpha^{-1}\beta\alpha = \gamma_{1}\cdots\gamma_{2g-2}, $$
$$\gamma_{i}\mbox{'s are simple transpositions, and}\ \langle\alpha, \beta, \gamma_{1}, \ldots, \gamma_{2g-2}\rangle\ \mbox{is transitive} \} .$$ 
There is an \emph{equivalence} relation $\sim$ among elements in $Cov$: 
$$ (\alpha, \beta, \gamma_{1}, \ldots, \gamma_{2g-2}) \sim (\alpha', \beta', \gamma'_{1}, \ldots, \gamma'_{2g-2} ) $$ if and only if there exists $\tau\in S_{d}$
such that $\tau (\alpha, \beta, \gamma_{1}, \ldots, \gamma_{2g-2}) \tau^{-1} = (\alpha', \beta', \gamma'_{1}, \ldots, \gamma'_{2g-2}). $
Moreover, let $N = |Cov/\sim|$ denote the cardinality of the set $Cov$ modulo the equivalence relation.  
\end{definition} 

Here the notation $\langle \cdot \rangle$ stands for the subgroup of $S_{d}$ generated by the elements inside. 
The transitivity condition is to make sure that the covers are connected. The equivalence relation amounts to relabeling the $d$ sheets of the cover. 
Two elements in $Cov$ correspond to isomorphic covers if and only if they are equivalent. The isomorphism between two covers 
$C\rightarrow E$ and $C'\rightarrow E$ means that there exists a commutative diagram as follows, 
$$\xymatrix{
C \ar[rr]^\phi \ar[dr] & & C' \ar[dl] \\
& E  &  }$$
where $\phi$ is an isomorphism between $C$ and $C'$. 

Then the number $N$ defined above counts the total number of degree $d$ genus $g$ connected covers of $E$ simply branched at $p_{1}, \ldots, p_{2g-2}$ up to isomorphism. Therefore, $W \rightarrow X$ is a finite map of degree $N$ unramified over the open subset $X^{0}$. It is possibly branched at $b_{i}$. 
To describe the local picture around $b_{i}$, it is necessary to understand which admissible covers appear when two branch points meet on the fiber over $b_{i}$. We can reorder the $2g-2$ sections and suppose locally
two sections $\Gamma_{1}$ and $\Gamma_{2}$ meet. The type of the corresponding admissible covers depend on the monodromy data $\gamma_{1}$ and $\gamma_{2}$. 
Since $\gamma_{1}$ and $\gamma_{2}$ are simple transpositions both permuting two letters of $\{1, \ldots, d\}$, their length-2 cycles can be the same, or contain only one common letter, 
or consist of four different letters. Denote the three cases by $\gamma_{1} = \gamma_{2},\  |\gamma_{1} \cap \gamma_{2}| = 1$ and $\gamma_{1} \cap \gamma_{2} = \emptyset$ 
respectively. We further define some subsets of $Cov$ based on the relation between $\gamma_{1}$ and $\gamma_{2}$ as follows. 

\begin{definition} For $1\leq h \leq [\frac{d}{2}]$, introduce following subsets of $Cov$. 
\label{subsets}
$$Cov_{0}:=\{ (\alpha, \beta, \gamma_{1}, \ldots, \gamma_{2g-2}) \in Cov\ |\ \gamma_{1} = \gamma_{2}, $$
$$\langle\alpha, \beta, \gamma_{3}, \ldots, \gamma_{2g-2} \rangle \ \mbox{acts transitively on}\  \{1, \ldots, d\} \};$$
$$Cov^{(h)}_{1}:= \{ (\alpha, \beta, \gamma_{1}, \ldots, \gamma_{2g-2}) \in Cov\ |\ \gamma_{1} = \gamma_{2}, $$
$$\langle\alpha, \beta, \gamma_{3}, \ldots, \gamma_{2g-2} \rangle\  \mbox{acts transitively on a partition} \ (h\ |\ d-h) \};$$ 
$$Cov_{2}:= \{ (\alpha, \beta, \gamma_{1}, \ldots, \gamma_{2g-2}) \in Cov\ |\  \gamma_{1} \cap \gamma_{2} = \emptyset \}; $$
$$Cov_{3}:= \{ (\alpha, \beta, \gamma_{1}, \ldots, \gamma_{2g-2}) \in Cov\ |\   |\gamma_{1} \cap \gamma_{2}| = 1 \}. $$

Moreover, let $$N_{0} = |Cov_{0} /\sim|, \ N^{(h)}_{1} = |Cov^{(h)}_{1}/\sim|, $$ 
$$N_{2} = |Cov_{2}/\sim|, \ N_{3} = |Cov_{3}/\sim|$$ 
denote the cardinalities of these subsets modulo the equivalence relation. 
\end{definition}   

The notation $(h\ |\ d-h)$ represents a partition of $\{1, \ldots, d\}$ into a subset of cardinality $h$ and its complement of cardinality $d-h$. 
If $ \langle\alpha, \beta, \gamma_{1}, \ldots, \gamma_{2g-2}\rangle$ acts transitively on $\{1, \ldots, d\}$ and $\gamma_{1} = \gamma_{2}$ 
are the same simple transposition, after taking away $\gamma_{1}$ and $\gamma_{2}$, the orbit of the permutations left over can at most break into
two orbits corresponding to a partition of type $(h\ |\ d-h)$. So the subsets $Cov_{0};\ Cov_{1}^{(h)}, 1\leq h \leq [\frac{d}{2}]$ all together fully cover the case $\gamma_{1} = \gamma_{2}$.  

Note that the equivalence relation is also well-defined for the above subsets. 
Furthermore, let $N_{1} = \sum_{h=1}^{[d/2]}N^{(h)}_{1}$. The subsets 
$Cov^{0}; \ Cov^{(h)}_{1}, 1\leq h \leq [\frac{d}{2}]; \ Cov_{2}; \ Cov_{3}$ 
yield a decomposition of $Cov$. Hence, we have the equality $ N = N_{0} + N_{1} + N_{2} + N_{3}.$ 

After the long and winding preparation, finally we can state one of our main results, a slope formula for $W$. 

\begin{theorem}
\label{slope}
The slope of $W$ is given by 
$$s(W) = \frac{72(N_{0}+N_{1})}{9(N_{0}+N_{1}) + N_{3}}. $$
\end{theorem}

Note that $s(W)$ only depends on the ratio of $N_{0}+N_{1}$ and $N_{3}$, which is independent of the base $X$ and the sections.  

The difficulty solving these $N_{i}$'s is essentially the same (if not harder!) as in \cite{HM1}. Unfortunately we have not been able to come up with a better method to analyze the 
asymptotic behavior of this slope formula for general $g$ and $d$. Nevertheless, we study the beginning case for $g=2$ and odd $d$ in detail. We have the following result. 

\begin{corollary}
\label{g2}
When $g=2$ and $d$ is odd, the slope $s(W_{d,2})$ is convergent to 5 as $d$ approaches infinity. 
\end{corollary}

Note that the sharp lower bound $s_{2} = 10$ on $\overline{\mathcal M}_{2}$ can be achieved by the family $Y$ we used in \cite{C}. So $W$ probably may not provide a better bound for slope than $Y$ does
in general. It would be interesting to calculate the slope for more one-parameter families in the Hurwitz space to compare with $W$ and $Y$.  

Another beginning case is when $d=2$. For those double covers of an elliptic curve, we have the following result. 
\begin{corollary}
\label{d2}
When $d=2$, the slope $s(W_{2,g})$ equals 8. Moreover, the base locus of an effective divisor with slope lower than 8 on \Mg must contain the locus of genus $g$ curves that 
admit a double cover to an elliptic curve. 
\end{corollary}

Switch our attention a little bit. The Hurwitz space $\overline{\mathcal H}_{d,g}$ of covers to elliptic curves could be reducible. So we have to study the irreducible components of $W_{d,g}$. 
Here we do it for a special case when the base $X = E$,  $\Gamma_{1},  \ldots, \Gamma_{2g-3}$ are distinct horizontal sections, and $\Gamma_{2g-2}$ is the diagonal of $E\times E$. See Figure \ref{EE}. 
 \begin{figure}[H]
    \centering
    \psfrag{B}{$E$}
    \psfrag{pi}{$p_{i}$}
    \psfrag{Ri}{$\Gamma_{i}$}
    \psfrag{Rg}{$\Gamma_{2g-2}$}
    \psfrag{EE}{$E\times E$}
     \includegraphics[scale=0.4]{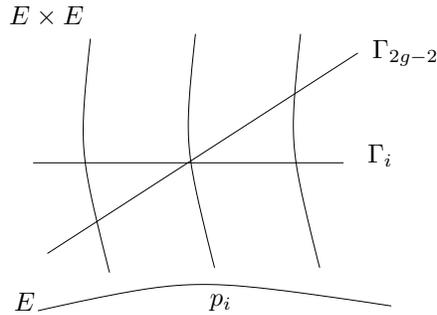}
    \caption{\label{EE} $E\times E$ with sections}
\end{figure}
This special case makes sense in that a section of $E\times X$ induces a map $X\rightarrow E$, which pulls back the covering family over $E\times E$ to a family over $E\times X$. 
We denote this special one-parameter family $W$ over $E$ by $W(E)$. Suppose $\Gamma_{2g-2}$ meets $\Gamma_{i}$ on a fiber over $p_{i}, i = 1, \ldots, 2g-3$. 
Then $W(E)\rightarrow E$ is a degree $N$ map unramified away from $p_{i}$'s. We want to study the local and global monodromy of $W(E) \rightarrow E$. 

By local monodromy, it means that if we go around $p_{i}$ once in its neighborhood along a loop $\sigma_{i}$ on the base $E$, there is an action permuting 
the $N$ sheets of $W(E)$ over the ending point $\sigma_{i}(0) = \sigma_{i}(1)$ of this loop. So the local monodromy reveals the ramification information of the map $W(E) \rightarrow E$ over $p_{i}$. Since it is local, the ramification information holds more generally for $W\rightarrow X$ over an arbitrary one-dimensional base $X$. The result is essentially the same as observed 
in \cite[Prop. 2.9.1]{HM1}.

\begin{proposition}
\label{ram}
$W\rightarrow X$ is triply ramified at a point corresponding to a local degeneration of covers in $Cov_{3}/\sim$, 
and is unramified everywhere else.  
\end{proposition}

Then we have a genus formula for $W$. 
\begin{corollary}
The genus of $W_{d,g}$ is given by $$g(W_{d,g}) = 1 + N(g(X) -1 ) + k N_{3}. $$ 
\end{corollary}

In section 2.1 we will explicitly describe which admissible covers appear as the degenerations of covers in those subsets
$Cov_{i}$.  

By global monodromy, we mean that for any path in $\pi_{1}(E_{b}, p_{1}, \ldots, p_{2g-3})$, if we go along it once, there is an action permuting the $N$ 
covers over $b$. Then two covers belong to the same irreducible component of $W(E)$ if and only if they can be sent to each other by a global monodromy action. 
To fully describe the global monodromy, it is sufficient to specify the actions corresponding to a basis of $\pi_{1}(E_{b}, p_{1}, \ldots, p_{2g-3})$. Since our application in mind 
is to move $p_{2g-2}$ along $E$, let $\eta_{i}$ be a closed path separating $p_{i}, \ldots, p_{2g-3}$ from the other marked points on the base $E$.  Then $\eta_{1}, \ldots, \eta_{2g-3}, \alpha, \beta$ 
generate $\pi_{1}(E_{b}, p_{1}, \ldots, p_{2g-3})$, cf. Figure \ref{eta}.  
\begin{figure}[H]
\centering
 \psfrag{E}{$E$}
    \psfrag{p1}{$p_{1}$}
    \psfrag{pi1}{$p_{i-1}$}
    \psfrag{pi}{$p_{i}$}
    \psfrag{pg}{$p_{2g-3}$}
    \psfrag{ita}{$\eta_{i}$}
    \psfrag{a}{$\alpha$}
    \psfrag{be}{$\beta$}
  \includegraphics[scale=0.5]{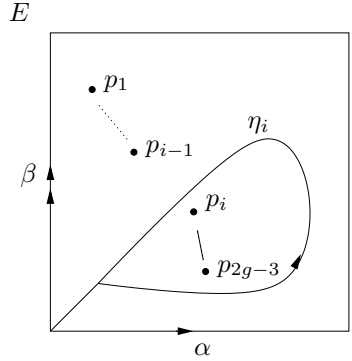}
    \caption{\label{eta} The path $\eta_{i}$ on $E$}
\end{figure}
\begin{definition} In the above set up, let $g_{i}, g_{\alpha}$ and $g_{\beta}$ 
correspond to $\gamma_{i}$, $\alpha$ and $\beta$ respectively, $1\leq i\leq 2g-3$, as the monodromy actions acting on $Cov/\sim$.   
\end{definition}

We state the monodromy criterion as follows. 

\begin{theorem}
\label{monodromy}
For a cover in $W(E)$ over $b$ corresponding to an element $ (\alpha, \beta, \gamma_{1}, \ldots, \gamma_{2g-2}) \in Cov/\sim$, the monodromy action $g_{i}$ acts like the following: 
$$g_{i}(\alpha) = \alpha, \ g_{i}(\beta) = \beta, \ g_{i} (\gamma_{j}) = \gamma_{j} \ \mbox{for}\  j < i, $$ 
$$g_{i}(\gamma_{j}) = \gamma_{2g-2}^{-1}\gamma_{j}\gamma_{2g-2}  \ \mbox{for}\  i\leq j \leq 2g-3, $$
$$g_{i}(\gamma_{2g-2}) = (\gamma_{i}\cdots \gamma_{2g-2})^{-1}\gamma_{2g-2}(\gamma_{i}\cdots \gamma_{2g-2}). $$ 
The monodromy actions $g_{\alpha}$ and $g_{\beta}$ act as follows: 
$$ g_{\alpha} (\alpha) = \alpha, \ g_{\alpha}(\beta) = \beta\gamma_{2g-2}, \ g_{\alpha}(\gamma_{2g-2}) = \alpha^{-1}\gamma_{2g-2}\alpha,$$ 
$$g_{\alpha}(\gamma_{j}) = \gamma_{2g-2}^{-1}\gamma_{j}\gamma_{2g-2}\  \mbox{for}\  j\leq 2g-3; $$
$$g_{\beta}(\alpha) = \alpha\gamma_{2g-2}^{-1}, \ g_{\beta}(\beta) = \beta, \ g_{\beta}(\gamma_{j}) = \gamma_{j} \ \mbox{for} \ j\leq 2g-3, $$ 
$$g_{\beta}(\gamma_{2g-2}) = (\beta\gamma_{1}\cdots\gamma_{2g-3})^{-1}\gamma_{2g-2} (\beta\gamma_{1}\cdots\gamma_{2g-3}). $$
\end{theorem}

Note that these monodromy actions preserve the equality $\beta^{-1}\alpha^{-1}\beta\alpha = \gamma_{1}\cdots\gamma_{2g-2}$ and the equivalence relation.
So they are well-defined on $Cov/\sim$. 

In practice it may not be easy to apply this monodromy criterion for large $g$ and $d$. Again we will only focus on the cases $g=2$ or $d=2$ to study the genus and components of $W(E)$. 

Back to our original motivation of constructing moving curves, reader might have already realized that a general genus $g$ curve cannot admit a cover to an elliptic curve, since 
there is only $(2g-2)$-dimensional freedom for the branch points, which is less than $3g-3$. So the deformation of $W$ may not cover an open set of $\overline{\mathcal M}_{g}$. 
However, for our purpose it suffices that the union of $W_{d,g}$ for all $d$ forms a
Zariski dense subset in $\overline{\mathcal M}_{g}$. Then for any effective divisor $D$ on $\overline{\mathcal M}_{g}$, there exists at least one $W_{d,g}$ not fully contained in $D$. The slope $s(W_{d,g})$ can still bound the slope of $D$. Actually, we have a density result as follows. 

\begin{theorem}
\label{density}
The union $\bigcup\limits_{d}^{\infty} W_{d,g}$ is Zariski dense in $\overline{\mathcal M}_{g}$. 
 \end{theorem}

Compared with the density result in \cite[Thm 1.7]{C}, the above is not surprising in the sense that covers in $Y$ can be regarded as degenerations of covers in $W$ by putting the branch points all together. 

The paper is organized in the following way. In section 2, we study the global geometry of $W$, namely, we prove the slope formula, the density result and the monodromy criterion. 
In section 3, we apply the general theory about $W$ to small genus or degree to derive various corollaries.  
Throughout the paper, we work over $\mathbb C$. A divisor on \Mg always denotes a $\mathbb Q$-Cartier divisor in the form $ a\lambda - \sum_{i=0}^{[g/2]} b_{i}\delta_{i}, a, b_{i} > 0$. We also assume that 
$g\geq 2$ all the time. 

{\bf Acknowledgements.} Special thanks are due to my advisor Joe Harris who always encouraged me and gave me many valuable ideas. 
I am also benefited from conversations with Gavril Farkas, Andrei Okounkov and Rahul Pandharipande, who listened to the results \cite{C} and 
suggested the author to calculate slopes for more families of covers. Finally I want to thank Martin Schmoll for pointing out some references to me. 

\section{The geometry of $W$}
In this section, we will focus on the slope, density and monodromy of $W$. Theorems \ref{slope}, \ref{monodromy} and \ref{density} will be verified one by one.   
\subsection{Slope} \ 

Recall our set up for $W$ in the introduction section. $W$ is the space of simply branched admissible covers of an elliptic curve over a one-dimensional base $X$. 
$X$ admits $2g-2$ sections $\Gamma_{1}, \ldots, \Gamma_{2g-2}$ to $E\times X$. No three sections meet at a common point. If two sections meet, they meet transversely and a fiber can have at most one such intersection point. 

Suppose two sections $\Gamma_{1}$ and $\Gamma_{2}$ meet at $p$ on a fiber $E$ over a base point $b_{l}\in X$. To describe the admissible cover, we need to blow up $E\times X$ at $p$, so the fiber becomes 
a nodal union of $E$ with an exceptional curve $\mathbb P^{1}$. The proper transforms $\widetilde{\Gamma}_{1}$ and $\widetilde{\Gamma}_{2}$
of $\Gamma_{1}$ and $\Gamma_{2}$ meet the exceptional $\mathbb P^{1}$ at $p_{1}$ and $p_{2}$ respectively. 
Suppose further that $E$ meets the other sections at $p_{3}, \ldots, p_{2g-2}$. See Figure \ref{X}. 
 \begin{figure}[H]
    \centering
    \psfrag{X}{$X$}
    \psfrag{E}{$E$}
    \psfrag{Ri}{$\widetilde{\Gamma}_{1}$}
    \psfrag{Rj}{$\widetilde{\Gamma}_{2}$}
    \psfrag{Rm}{$\Gamma_{m}$}
    \psfrag{pi}{$p_{1}$}
    \psfrag{pj}{$p_{2}$}
    \psfrag{p}{$p$}
    \psfrag{qm}{$p_{m}$}
    \psfrag{P}{$\mathbb P^{1}$}
    \psfrag{bl}{$b_{l}$}
     \includegraphics[scale=0.4]{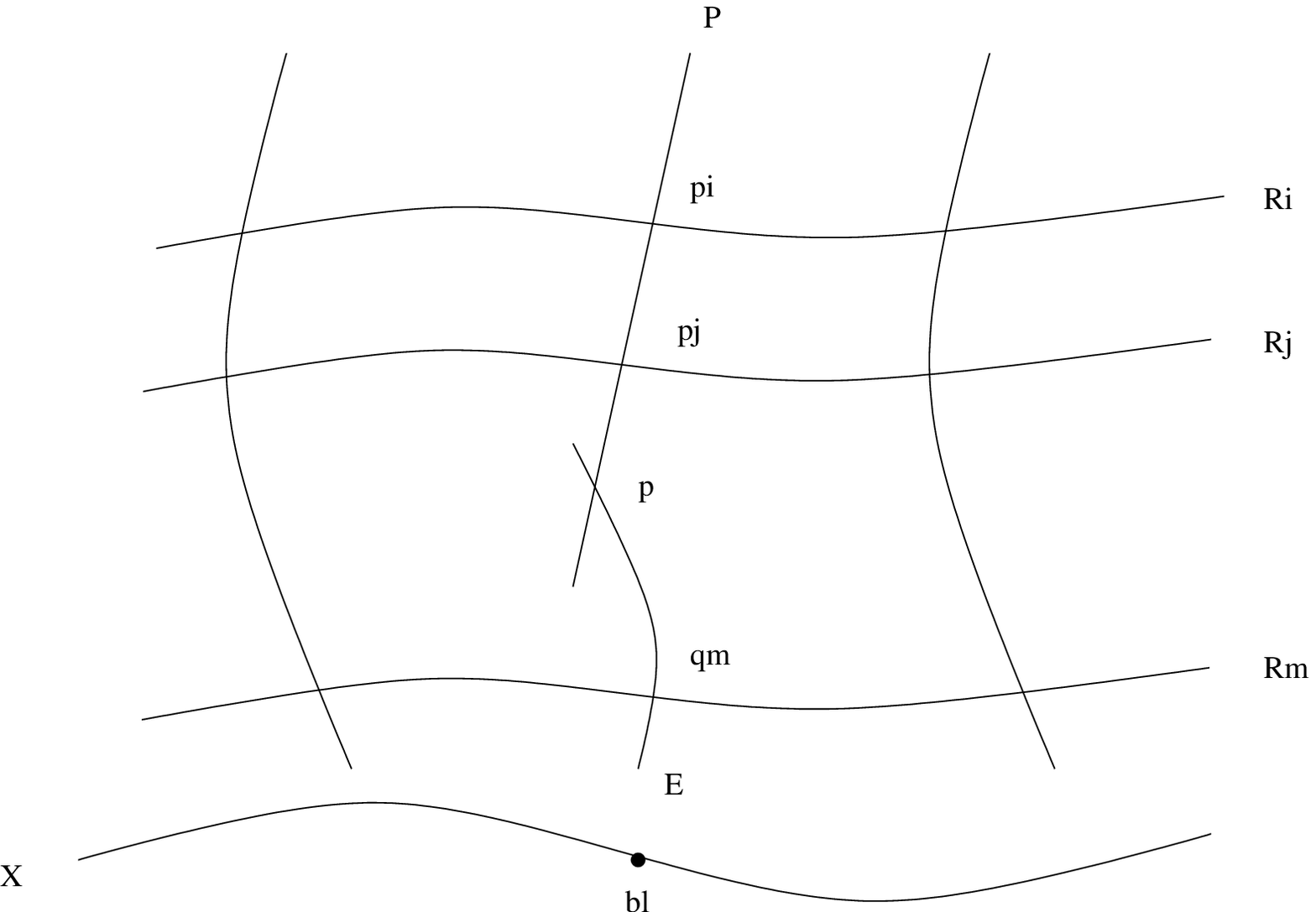}
    \caption{\label{X} The blow-up of $E\times X$ at $p$}
\end{figure}
A smooth cover in $W$ over a general point $b\in X$ would degenerate to a singular admissible cover of $E\cup\mathbb P^{1}$ when $b$ approaches to $b_{l}$. The type of the admissible cover depends on 
the combinatorial type of the element $ (\alpha, \beta, \gamma_{1}, \ldots, \gamma_{2g-2}) \in Cov$ associated to the smooth cover. More precisely, we have already decomposed $Cov$ using 
the subsets $Cov_{0}; \ Cov^{(h)}_{1}, 1\leq h \leq [\frac{d}{2}]; \ Cov_{2};\  Cov_{3}$, cf. Definition \ref{subsets}. The type of the admissible cover is uniquely determined by the subset 
$ (\alpha, \beta, \gamma_{1}, \ldots, \gamma_{2g-2})$ belongs to. We will follow a similar description as
\cite{HM1} to draw the degenerate admissible covers and illustrate their stable limits. 

The admissible covers in Figure \ref{Cov0} correspond to the degeneration of covers in $Cov_{0}$. $C_{0}$ is a smooth genus $g-1$ curve admitting a degree $d$ cover to $E$ simply branched at
$p_{3}, \ldots, p_{2g-2}$. There are $d-1$ rational curves attached to $C_{0}$. The first one meets $C_{0}$ at 
two distinct points in the pre-image of $p$ and admits a double cover to the exceptional $\mathbb P^{1}$ simply branched at $p_{1}, p_{2}$. The other $d-2$ rational curves are all tails each 
of which meets $C_{0}$ at one point in the pre-image of $p$ and maps to the exceptional $\mathbb P^{1}$ isomorphically. The stable limit of this covering is an irreducible one-nodal curve of 
geometric genus $g-1$ by identifying two points of $C_{0}$.  
\begin{figure}[H]
    \centering
    \psfrag{E}{$E$}
    \psfrag{P}{$\mathbb P^{1}$}
    \psfrag{e}{$(d-2)\  \mathbb P^{1}$'s}
    \psfrag{C0}{$C_{0}$}
    \psfrag{p}{$p$}
    \psfrag{pi}{$p_{1}$}
    \psfrag{pj}{$p_{2}$}
    \psfrag{q1}{$p_{3}$}
    \psfrag{qg}{$p_{2g-2}$}
    \includegraphics[scale=0.4]{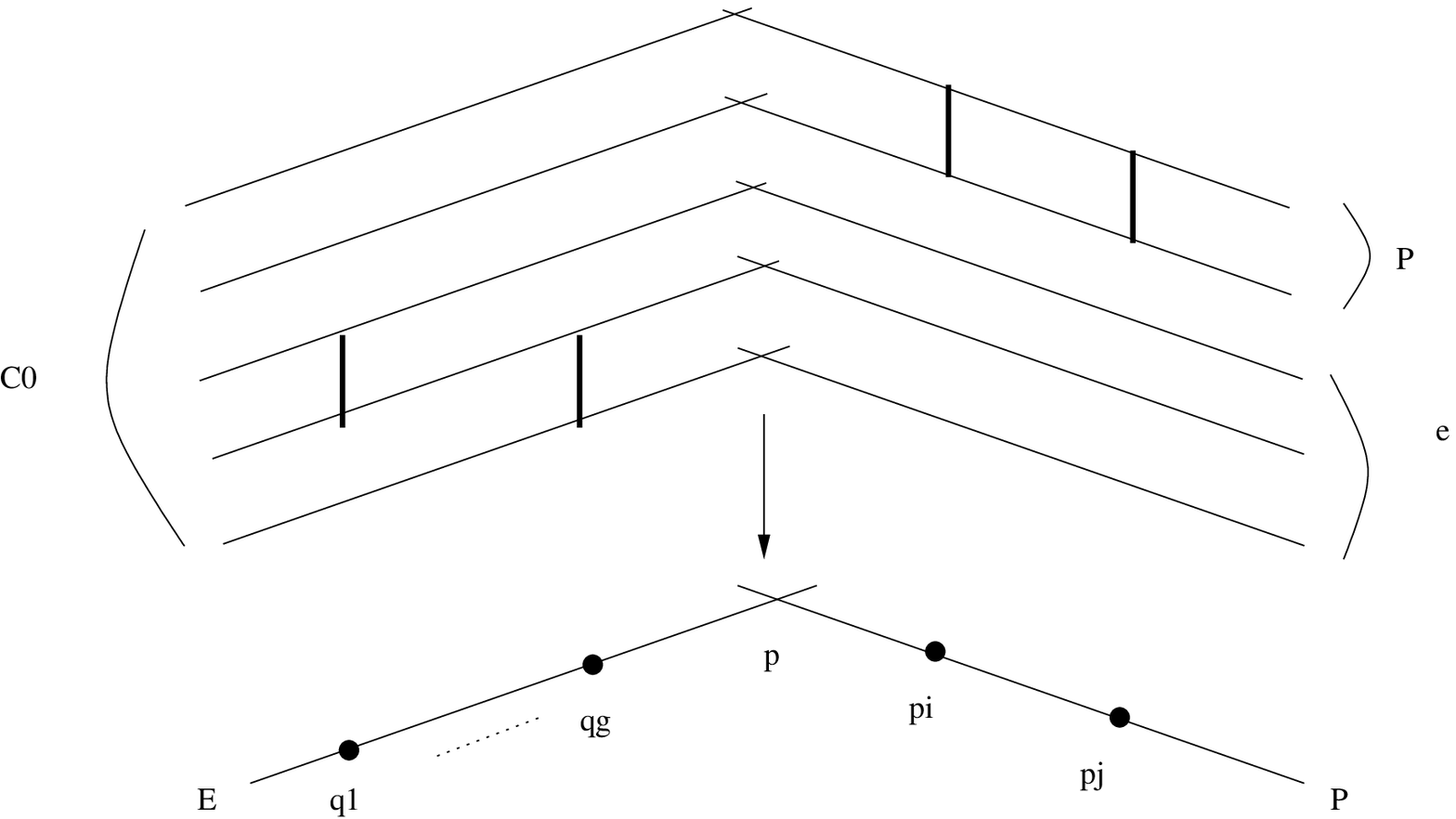}
    \caption{\label{Cov0} Degeneration of covers in $Cov_{0}$}
\end{figure}
The admissible covers in Figure \ref{Cov1} correspond to the degeneration of covers in $Cov^{(h)}_{1}, 1\leq h \leq [\frac{d}{2}]$. By the definition of $Cov^{(h)}_{1}$, without loss of the generality we can assume that 
$\langle\alpha, \beta, \gamma_{3}, \ldots, \gamma_{2g-2} \rangle$ acts transitively on $\{1,\ldots, h\}$ and $\{h+1, \ldots, d\}$ respectively.   
Suppose there are $m$ cycles out of $(\gamma_{3}, \ldots, \gamma_{2g-2})$ whose letters are contained in $\{1,\ldots, h\}$, and $2g-4-m$ cycles 
whose letters are in $\{h+1, \ldots, d\}$. Then the genera of $C_{1}$ and $C'_{1}$ satisfy 
$2g(C_{1}) - 2 = m, \ 2g(C'_{1})-2 = 2g-4-m.$ 
It is also not hard to directly check that $m$ must be even in this case. 
$C_{1}$ admits a degree $h$ map to $E$ simply branched at 
$m$ intersection points of $E$ with the corresponding $m$ sections. Similarly, $C'_{1}$ admits a degree $d-h$ map to $E$ simply branched at the other $2g-4-m$ intersection points on $E$. 
There is a rational curve connecting $C_{1}$ and $C'_{1}$ at two points in the pre-image of $p$. Moreover, it admits a double cover to the exceptional $\mathbb P^{1}$ simply branched at $p_{1}, p_{2}$. 
There are also $h-1$ rational tails attached to $C_{1}$ and $d-h-1$ rational tails attached to $C'_{1}$ at the other points in the pre-image of $p$. Each rational tail maps to the 
exceptional $\mathbb P^{1}$ isomorphically. The stable limit of this covering is a reducible one-nodal union of $C_{1}$ and $C'_{1}$. Note that $g(C_{1}) + g(C'_{1}) = g$. 
\begin{figure}[H]
    \centering
    \psfrag{E}{$E$}
    \psfrag{C1}{$C_{1}$}
    \psfrag{C2}{$C'_{1}$}
    \psfrag{P}{$\mathbb P^{1}$}
    \psfrag{h}{$(h-1)\  \mathbb P^{1}$'s}
    \psfrag{dh}{$(d-h-1)\ \mathbb P^{1}$'s}
   \psfrag{p}{$p$}
     \psfrag{pi}{$p_{1}$}
    \psfrag{pj}{$p_{2}$}
    \psfrag{q1}{$p_{3}$}
    \psfrag{qg}{$p_{2g-2}$}
    \includegraphics[scale=0.4]{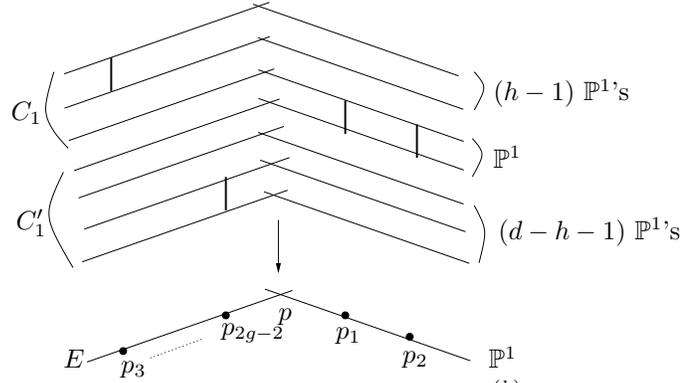}
    \caption{\label{Cov1} Degeneration of covers in $Cov^{(h)}_{1}$}
\end{figure}
The admissible covers in Figure \ref{Cov2} correspond to the degeneration of covers in $Cov_{2}$. $C_{2}$ is a smooth genus $g$ curve that admits a degree $d$ cover to $E$ simply branched at 
$p_{3}, \ldots, p_{2g-2}$ as well as simply ramified at two points $x, y$ over $p$. Attach to $C_{2}$ a rational curve that admits a double cover of the exceptional $\mathbb P^{1}$ by identifying 
one of the ramification points with $x$. In the same way, attach another rational curve to $C_{2}$ at $y$. Locally around $x$ and $y$, the covering map
to $p$ is given by $(u,v) \rightarrow (u^{2}, v^{2}).$ 
Finally, attach $d-4$ rational tails to $C_{2}$ at the other points in the pre-image of $p$. 
Each rational tail admits an isomorphic map to the exceptional $\mathbb P^{1}$. The stable limit of this covering is just $C_{2}$. 
\begin{figure}[H]
    \centering
    \psfrag{E}{$E$}
    \psfrag{D}{$C_{2}$}
    \psfrag{Q}{$(d-4)\ \mathbb P^{1}$'s}
    \psfrag{P}{$\mathbb P^{1}$}
    \psfrag{x}{$x$}
    \psfrag{y}{$y$}
     \psfrag{p}{$p$}
     \psfrag{pi}{$p_{1}$}
    \psfrag{pj}{$p_{2}$}
    \psfrag{q1}{$p_{3}$}
    \psfrag{qg}{$p_{2g-2}$}
    \includegraphics[scale=0.4]{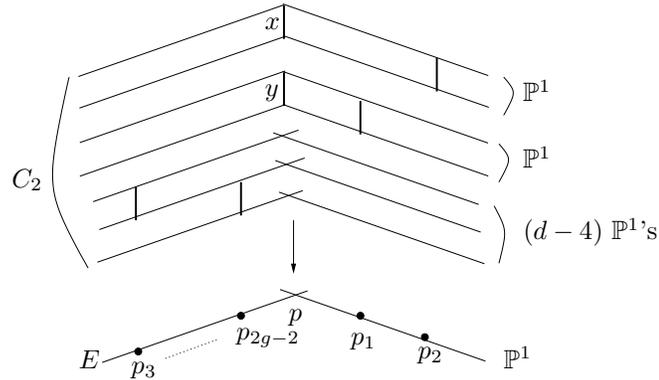}
    \caption{\label{Cov2} Degeneration of covers in $Cov_{2}$}
\end{figure}
Finally for covers in $Cov_{3}$, their degeneration corresponds to the admissible covers as Figure \ref{Cov3}. $C_{3}$ is a smooth genus $g$ curve that admits a degree $d$ cover to $E$, simply branched at 
$p_{3}, \ldots, p_{2g-2}$ and triply ramified at a point $x$ over $p$. Take a rational curve admitting a triple cover of the exceptional $\mathbb P^{1}$ that has a triple ramification point $y$ over $p$ 
and two other simple ramification points over $p_{1}, p_{2}$. Attach this rational curve to $C_{3}$ by gluing $y$ to $x$. Locally around $x=y$, the covering map to $p$ is given by $(u,v)\rightarrow (u^{3}, v^{3})$. 
Then attach $d-3$ rational tails to $C_{3}$ at the other points in the pre-image of $p$.
Each rational tail maps isomorphically to the exceptional $\mathbb P^{1}$. The stable limit of this covering is just $C_{3}$. 
\begin{figure}[H]
    \centering
    \psfrag{E}{$E$}
    \psfrag{F}{$C_{3}$}
    \psfrag{Q}{$(d-3)\ \mathbb P^{1}$'s}
    \psfrag{P}{$\mathbb P^{1}$}
    \psfrag{x}{$x$}
    \psfrag{y}{$y$}
     \psfrag{p}{$p$}
     \psfrag{pi}{$p_{1}$}
    \psfrag{pj}{$p_{2}$}
    \psfrag{q1}{$p_{3}$}
    \psfrag{qg}{$p_{2g-2}$}
    \includegraphics[scale=0.4]{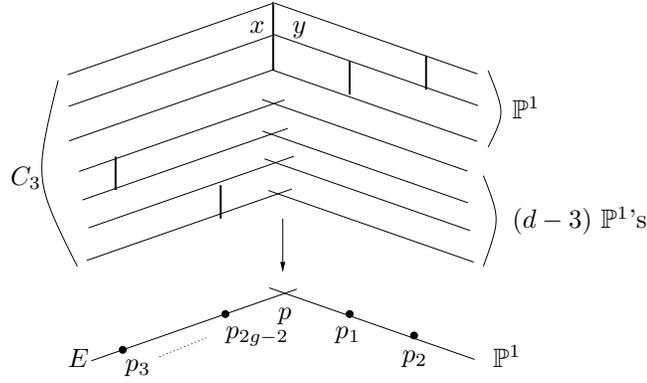}
    \caption{\label{Cov3} Degeneration of covers in $Cov_{3}$}
\end{figure}
To calculate the slope of $W$, we want to know how $W$ intersects the boundary of $\overline{\mathcal M}_{g}$. 

\begin{proposition}
\label{Delta}
 We have the following intersection numbers: 
 $$ W\ldotp \delta_{0} = 2kN_{0}, \ W\ldotp (\delta_{1}+\cdots + \delta_{[\frac{g}{2}]}) = 2k N_{1}. $$
 In particular, the intersection of $W$ with the total boundary is 
 $$W\ldotp \delta = 2k (N_{0} + N_{1}). $$
\end{proposition}

\begin{proof}
These intersection numbers can be read off using the information of the above admissible covers.
 
For covers in $Cov_{0}$, the degeneration of each contributes 2 to $W\ldotp \delta_{0}$, since the stable limit only lies in $\Delta_{0}$ 
and the rational curve that admits a double cover to the exceptional $\mathbb P^{1}$ has self-intersetion $-2$. 

For covers in $Cov^{(h)}_{1}, 1\leq h \leq [\frac{d}{2}]$, the degeneration of each contributes 2 to $W\ldotp (\delta_{1}+\cdots + \delta_{[\frac{g}{2}]})$, 
since the stable limit is a reducible one-nodal union of two smooth curves and the rational curve that admits a double cover also has self-intersetion $-2$. 

For covers in $Cov_{2}$ and $Cov_{3}$, their degenerations do not contribute to the intersection with the boundary components, since the stable limits are smooth 
genus $g$ curves. 

Now, the $2g-2$ sections meet pair-wisely at $k$ points, namely, $\sum_{i < j} \Gamma_{i}\ldotp \Gamma_{j} = k$. The proposition follows immediately. 
\end{proof}

Next, we will work out the degree of $\lambda$ on $W$. 

\begin{proposition}
\label{Lambda}
In the above set up, we have 
$$ W\ldotp \lambda = k \Big( \frac{N_{0} + N_{1}}{4} + \frac{N_{3}}{36} \Big). $$
\end{proposition}

\begin{proof}
Let Bl($E\times X$) be the blow-up of $E\times X$ at the $k$ intersection points of $\Gamma_{1}, \ldots, \Gamma_{2g-2}$. Then there are $k$ special fibers of 
Bl($E\times X$) over $X$, each of which is a union of $E$ with an exceptional $\mathbb P^{1}$. Let $\widetilde{\Gamma}_{i}$ denote the proper transform of $\Gamma_{i}$. 
Also suppose that a section $\Gamma_{i}$ meets other sections $k_{i}$ times. Then $\sum_{i=1}^{2g-2}k_{i} = 2k$. 
Pull back Bl($E\times X$) via the map $W\rightarrow X$ to a surface $S$ over $W$. We would like to form a universal admissible covering as follows,
$$\xymatrix{
T \ar[rr] \ar[dr] & & S \ar[dl] \\
& W &  }$$
such that over a singular fiber of $S$, a suitable degenerate admissible cover as described before would appear as the corresponding fiber of $T$. 
This is not possible without a \emph{base change}, already noticed in \cite{HM1}, due to the fact that $\overline{\mathcal H}_{d,g} $ is a coarse moduli space rather than fine.

In particular, the admissible cover corresponding to the degeneration of covers in $Cov_{2}$, cf. Figure \ref{Cov2}, 
 has automorphisms that come from 
the involution of the two rational curves admitting double covers of the exceptional $\mathbb P^{1}$. Locally we have to make a degree 2 base change such that 
around the node $p$ of the singular fiber, the surface $S$ is given by local equation $zw = s^{2}$. Upstairs there are two nodes $x$ and $y$ simply ramified over $p$, 
and their local equations for $T$ look like $uv = s$. The map is given by $z = u^{2}, w = v^{2}$. Notice that a singularity of type $A_{1}$ arises after the base change. This implies again why we count 2 for its contribution to the intersection of $W$ with the boundary of $\overline{\mathcal M}_{g}$.  

An \emph{ad hoc} analysis for the local information around a degenerate cover has been studied intensively, cf. \cite[Thm. 2.15]{HM1}. Of course one can follow closely their method to derive our desired 
calculation. However, here we will take a slightly different approach, by passing to a finite cover of the Hurwitz space, over which a universal family stands. This kind of calculation is less cumbersome, and can 
be generalized to more complicated ramification types, e.g., cf. \cite[2.1]{C} for an application. 

Since the slope as a quotient is invariant under a finite base change, we make a degree $n$ base change $\widetilde{W}\rightarrow W$ for a big enough $n$ to realize the desired  
universal covering: 
 $$\xymatrix{
\widetilde{T} \ar[rr]^{\phi} \ar[dr] & & \widetilde{S} \ar[dl] \\
& \widetilde{W} &  }$$

We want to show that 
$$ \widetilde{W}\ldotp \lambda = nk \Big( \frac{N_{0} + N_{1}}{4} + \frac{N_{3}}{36} \Big), $$
from which the proposition follows right away. 

We use the well-known relation $12\lambda = \delta_{\widetilde{T}} + \kappa$ for the family $\widetilde{T}\rightarrow \widetilde{W}$. 
$ \delta_{\widetilde{T}}\ldotp \widetilde{T}$ counts the total number of nodes in the singular fibers of $\widetilde{T}$. By enumerating the nodes of 
those degenerate admissible covers, we have 
$$ \delta_{\widetilde{T}}\ldotp \widetilde{T} = n k \big( d(N_{0} + N_{1}) + (d-4 + \frac{1}{2} + \frac{1}{2})N_{2} + (d-3 + \frac{1}{3})N_{3} \big) = nk\big(dN - 3N_{2} - \frac{8}{3}N_{3}\big). $$
Note that if locally around a node the map is given by $(u, v) \rightarrow (u^{m}, v^{m})$, then we count it with weight $\frac{1}{m}$ since $m$ nearby sheets in $\widetilde{W}$ corresponding to smooth covers 
would approach together. 

For $\kappa\ldotp \widetilde{W} = \omega^{2}_{\widetilde{T}/\widetilde{W}}$ where $\omega$ denotes the first Chern class of the relative dualizing sheaf, we will use the branched cover 
$\widetilde{T}\rightarrow \widetilde{S}$. The calculation eventually boils down to a calculation on Bl($E\times X$). 

By Riemann-Hurwitz, $\omega_{\widetilde{T}/\widetilde{W}} = \phi^{*}\omega_{\widetilde{S}/\widetilde{W}} + \sum\limits_{i=1}^{2g-2}D_{i}$, where 
$D_{i}$ is the ramification section on $\widetilde{T}$ whose image on $\widetilde{S}$ corresponds to the pull back of $\widetilde{\Gamma}_{i}$. Hence, 
$$\omega^{2}_{\widetilde{T}/\widetilde{W}} = (\phi^{*}\omega_{\widetilde{S}/\widetilde{W}})^{2} + 2 \big( \sum\limits_{i=1}^{2g-2}\phi^{*}\omega_{\widetilde{S}/\widetilde{W}}\ldotp D_{i}\big) +\big(\sum\limits_{i=1}^{2g-2}D_{i}\big)^{2}. $$

We have
$$(\phi^{*}\omega_{\widetilde{S}/\widetilde{W}})^{2} = d (\omega_{\widetilde{S}/\widetilde{W}})^{2} = n d (\omega_{S/W})^2 = n d N (\omega_{Bl(E\times X)/X})^{2} = - k n d N. $$

Moreover, 
$$\phi^{*}\omega_{\widetilde{S}/\widetilde{W}}\ldotp D_{i} = \omega_{\widetilde{S}/\widetilde{W}}\ldotp (\phi_{*}D_{i}) = - n N \widetilde{\Gamma}_{i}^{2} = - n N (\Gamma_{i}^{2}-k_{i}), $$
and 
$$ D_{i}^{2} = \frac{1}{2} n N \widetilde{\Gamma}_{i}^{2} = \frac{1}{2} n N(\Gamma_{i}^{2} - k_{i}), \ D_{i}\ldotp D_{j} = 0, \ i\neq j.$$ 

Therefore, $$ \omega^{2}_{\widetilde{T}/\widetilde{W}} = - kndN + 2nN\Big(\sum\limits_{i=1}^{2g-2}(k_{i} - \Gamma_{i}^{2}) \Big)+ \frac{1}{2} nN \Big(\sum\limits_{i=1}^{2g-2}(\Gamma_{i}^{2}-k_{i})\Big) $$
 $$ = n N \Big( - kd + 3k - \frac{3}{2}\sum\limits_{i=1}^{2g-2}\Gamma_{i}^{2}\Big) = n N k (3-d),$$ 
 since $\Gamma_{i}^{2} = 0 $ on $E\times X$. 
 
 Plugging the above into the relation $\widetilde{W}\ldotp \lambda = \frac{1}{12}(\widetilde{W}\ldotp \delta_{\widetilde{T}}+\widetilde{W}\ldotp\kappa), $ we get the desired equality. 
\end{proof}

Now Theorem \ref{slope} follows immediately from Proposition \ref{Delta} and \ref{Lambda}. 

\begin{remark}
From the slope formula in Theorem \ref{slope}, we notice that $s(W)$ is independent of the base $X$ and the sections. Actually this is not surprising. The union of $2g-2$ marked points corresponds to a degree $2g-2$ divisor of $E$. Let us forget the order among the $2g-2$ sections for the time being. Then $X$ can be regarded as an effective curve in $E_{2g-2}$, the symmetric product of $E$. Note that $E_{2g-2}$ is a projective bundle over Pic($E) = E$. In particular, its cone of effective curves is generated by \emph{two} classes. Take one curve class by moving a divisor on $E$ using the translation of $E$. The one-parameter family of its pre-image in the Hurwitz space maps trivially to a point in $\overline{\mathcal M}_{g}$, since all the covering curves are isomorphic in this family. Therefore, the slope does not depend on which base $X$ or sections we choose. In fact, in the next subsection we will focus on the 
most genuine model when $X = E$ along with $2g-3$ horizontal sections and the diagonal of $E\times E$ to study the monodromy actions. 
\end{remark}

\begin{remark}
Even when $W$ is reducible, the slope formula can be applied to each component of $W$. A component of $W$ corresponds to a subset of $Cov/\sim$ invariant under 
the monodromy actions, cf. next subsection. We can decompose this subset exactly as what we did for $Cov$ and re-enumerate those $N_{i}$'s. Based on the above remark, 
we suspect that each component of $W$ should also have the same slope. 
\end{remark}

\subsection{Monodromy}\ 

In this section we will prove Theorem \ref{monodromy} about the monodromy of $W(E)\rightarrow E$. Recall that 
$W(E)$ is the space of degree $d$ genus $g$ admissible covers of $E$ simply branched at $2g-2$ points $p_{1}, \ldots, p_{2g-2}$, where the first $2g-3$ points are fixed and the last one is moving on $E$. For a general base point $b\in E$ and 
the paths $\alpha, \beta, \eta_{1}, \ldots, \eta_{2g-3}$ that generate $\pi_{1}(E_{b}, p_{1}, \ldots, p_{2g-3})$, cf. 
Figure \ref{eta}, we want to study the monodromy actions $g_{\alpha}, g_{\beta}, g_{1}, \ldots g_{2g-3}$ associated to 
these paths. Let $(\alpha, \beta, \gamma_{1}, \ldots, \gamma_{2g-2})\in Cov$ correspond to a cover in $W(E)$ over $b$. 

\begin{proposition}
\label{mi}
Going along $\eta_{i}$ once, $1\leq i \leq 2g-3$, the monodromy action $g_{i}$ acts on $(\alpha, \beta, \gamma_{1}, \ldots, \gamma_{2g-2})\in Cov$ as follows: 
$$g_{i}(\alpha) = \alpha, \ g_{i}(\beta) = \beta, \ g_{i} (\gamma_{j}) = \gamma_{j} \ \mbox{for}\  j < i, $$ 
$$g_{i}(\gamma_{j}) = \gamma_{2g-2}^{-1}\gamma_{j}\gamma_{2g-2}  \ \mbox{for}\  i\leq j \leq 2g-3, $$
$$g_{i}(\gamma_{2g-2}) = (\gamma_{i}\cdots \gamma_{2g-2})^{-1}\gamma_{2g-2}(\gamma_{i}\cdots \gamma_{2g-2}).$$ 
\end{proposition}

\begin{proof}
Since $p_{2g-2}$ is the moving branch point corresponding to the diagonal section of $E\times E$, over $p_{i}$ on the base $E$, 
the fiber has an intersection point of $\Gamma_{i}$ with $\Gamma_{2g-2}$. Going along the path $\eta_{i}$ once, it amounts to 
circling $p_{2g-2}$ around $p_{i}, \ldots, p_{2g-3}$ once while keeping all the paths $\alpha, \beta, \gamma_{j}$'s in a \emph{relatively fixed} position. 
Figure \ref{gi} illustrates what happens when $p_{2g-2}$ moves along the dashed circle homotopic to $\eta_{i}$.  
\begin{figure}[H]
    \centering
    \psfrag{p1}{$p_{1}$}
    \psfrag{pi1}{$p_{i-1}$}
    \psfrag{pi}{$p_{i}$}
    \psfrag{pg3}{$p_{2g-3}$}
    \psfrag{pg2}{$p_{2g-2}$}
    \psfrag{b}{$b$}
     \psfrag{r1}{$\gamma_{1}$}
     \psfrag{ri1}{$\gamma_{i-1}$}
     \psfrag{ri}{$\gamma_{i}$}
    \psfrag{rg3}{$\gamma_{2g-3}$}
    \psfrag{rg2}{$\gamma_{2g-2}$}
    \psfrag{s1}{$\gamma'_{1}$}
    \psfrag{si1}{$\gamma'_{i-1}$}
    \psfrag{si}{$\gamma'_{i}$}
    \psfrag{sg3}{$\gamma'_{2g-3}$}
    \psfrag{sg2}{$\gamma'_{2g-2}$}
    \psfrag{gi}{$g_{i}$}
    \psfrag{a}{$\alpha$}
    \psfrag{be}{$\beta$}
    \psfrag{a1}{$\alpha'$}
    \psfrag{b1}{$\beta'$}
    \psfrag{eta}{$\eta_{i}$}
    \includegraphics[scale=0.5]{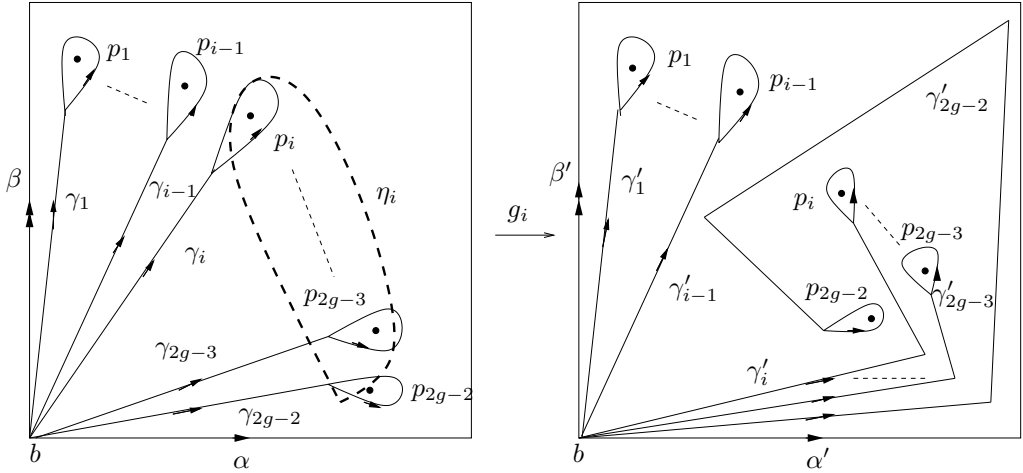}
    \caption{\label{gi} The action of $g_{i}$}
\end{figure}
Note that in $\pi_{1}(E_{b}, p_{1}, \ldots, p_{2g-3}),$ we have the following relations to express the new paths by the original paths: 
$$\gamma'_{j} \sim \gamma_{j},\  j<i, $$
$$\gamma'_{j} \sim  \gamma_{2g-2}^{-1}\gamma_{j}\gamma_{2g-2}, \ i\leq j \leq 2g-3, $$
$$\gamma'_{2g-2} \sim (\gamma_{i}\cdots \gamma_{2g-2})^{-1}\gamma_{2g-2}(\gamma_{i}\cdots \gamma_{2g-2}).$$ 
Moreover, $\alpha'\sim \alpha$ and $\beta'\sim\beta$ do not change under $g_{i}$. These relations are exactly the same as in 
the above proposition. 
\end{proof}

\begin{proposition}
\label{ma}
Going along $\alpha$ once, the monodromy action $g_{\alpha}$ acts on 
$(\alpha, \beta, \gamma_{1}, \ldots, \gamma_{2g-2})\in Cov$ as follows: 
$$ g_{\alpha} (\alpha) = \alpha, \ g_{\alpha}(\beta) = \beta\gamma_{2g-2}, \ g_{\alpha}(\gamma_{2g-2}) = \alpha^{-1}\gamma_{2g-2}\alpha,$$ 
$$g_{\alpha}(\gamma_{j}) = \gamma_{2g-2}^{-1}\gamma_{j}\gamma_{2g-2}\  \mbox{for}\  j\leq 2g-3.$$
\end{proposition}

\begin{proof}
Again we draw what happens to those paths as $p_{2g-2}$ moves along the dashed path homotopic to $\alpha$, cf. Figure \ref{ga}. 
\begin{figure}[h]
    \centering
    \psfrag{p1}{$p_{1}$}
    \psfrag{p2}{$p_{2}$}
    \psfrag{pg}{$p_{2g-2}$}
    \psfrag{b}{$b$}
     \psfrag{r1}{$\gamma_{1}$}
     \psfrag{r2}{$\gamma_{2}$}
     \psfrag{rg}{$\gamma_{2g-2}$}
    \psfrag{s1}{$\gamma'_{1}$}
    \psfrag{s2}{$\gamma'_{2}$}
    \psfrag{sg}{$\gamma'_{2g-2}$}
    \psfrag{ga}{$g_{\alpha}$}
    \psfrag{a}{$\alpha$}
    \psfrag{be}{$\beta$}
    \psfrag{a1}{$\alpha'$}
    \psfrag{b1}{$\beta'$}
    \includegraphics[scale=0.6]{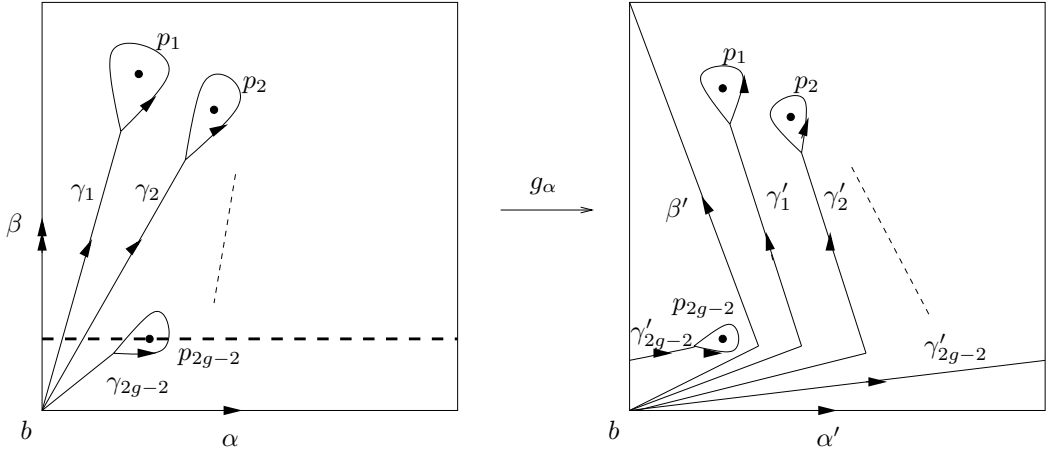}
    \caption{\label{ga} The action of $g_{\alpha}$}
\end{figure}
We have the following relations to express the new paths by the original paths: 
$$ \alpha' \sim\alpha, \ \beta' \sim \beta\gamma_{2g-2}, \ \gamma'_{2g-2} \sim \alpha^{-1}\gamma_{2g-2}\alpha,$$ 
$$\gamma'_{j} \sim \gamma_{2g-2}^{-1}\gamma_{j}\gamma_{2g-2}, \  j\leq 2g-3. $$

These relations imply the above proposition. 
\end{proof}

\begin{proposition}
\label{mb}
Going along $\beta$ once, the monodromy action $g_{\beta}$ acts on 
$(\alpha, \beta, \gamma_{1}, \ldots, \gamma_{2g-2})\in Cov$ as follows: 
$$g_{\beta}(\alpha) = \alpha\gamma_{2g-2}^{-1}, \ g_{\beta}(\beta) = \beta, \ g_{\beta}(\gamma_{j}) = \gamma_{j} \ \mbox{for} \ j\leq 2g-3, $$ 
$$g_{\beta}(\gamma_{2g-2}) = (\beta\gamma_{1}\cdots\gamma_{2g-3})^{-1}\gamma_{2g-2} (\beta\gamma_{1}\cdots\gamma_{2g-3}). $$
\end{proposition}

\begin{proof}
This time the action changes the paths as in Figure \ref{gb} when $p_{2g-2}$ moves along the dashed path homotopic to $\beta$. 
\begin{figure}[h]
    \centering
    \psfrag{p1}{$p_{1}$}
    \psfrag{p2}{$p_{2}$}
    \psfrag{pg}{$p_{2g-2}$}
    \psfrag{b}{$b$}
     \psfrag{r1}{$\gamma_{1}$}
     \psfrag{r2}{$\gamma_{2}$}
     \psfrag{rg}{$\gamma_{2g-2}$}
    \psfrag{s1}{$\gamma'_{1}$}
    \psfrag{s2}{$\gamma'_{2}$}
    \psfrag{sg}{$\gamma'_{2g-2}$}
    \psfrag{gb}{$g_{\beta}$}
    \psfrag{a}{$\alpha$}
    \psfrag{be}{$\beta$}
    \psfrag{a1}{$\alpha'$}
    \psfrag{b1}{$\beta'$}
    \includegraphics[scale=0.5]{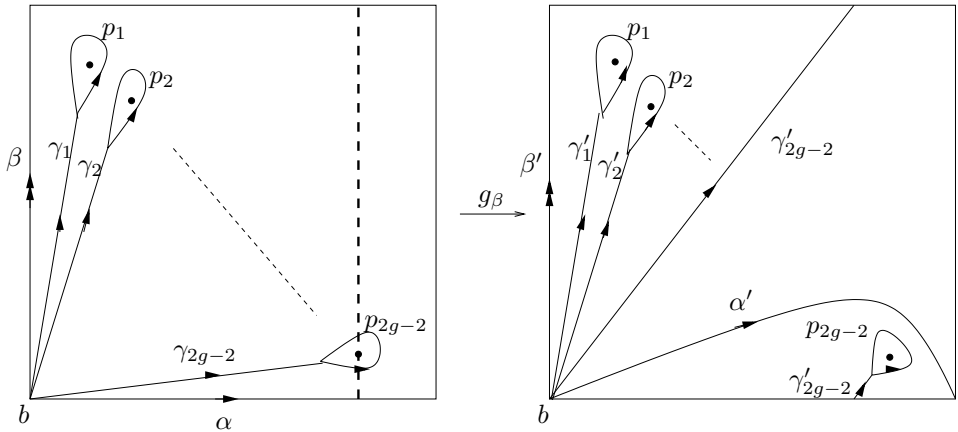}
    \caption{\label{gb} The action of $g_{\beta}$}
\end{figure}
We have the following relations to express the new paths by the original paths: 
$$\alpha' \sim \alpha\gamma_{2g-2}^{-1}, \ \beta' \sim \beta, \ \gamma'_{j} \sim \gamma_{j}, \ j\leq 2g-3, $$ 
$$\gamma'_{2g-2} \sim (\beta\gamma_{1}\cdots\gamma_{2g-3})^{-1}\gamma_{2g-2} (\beta\gamma_{1}\cdots\gamma_{2g-3}), $$
which implies the proposition immediately. 
\end{proof}

Now Theorem \ref{monodromy} is simply the combination of Proposition \ref{mi}, \ref{ma} and \ref{mb}.

\subsection{Density}\ 

Recall that we want to produce sufficiently many curves in \Mg such that an effective divisor cannot contain all of them. 
One way is to construct moving curves whose deformations cover an open subset of $\overline{\mathcal M}_{g}$.  

The one-parameter family of covers $W_{d,g}$ may not be fully moving since its visible variation depends only on the $2g-2$ 
branch points on $E$. However, if we vary the degree $d$ of the covers, the union of the countably many 
$W_{d,g}$ is Zariski dense in $\overline{\mathcal M}_{g}$, as stated in Theorem \ref{density}. The proof of Theorem \ref{density} is essentially the same as \cite[2.3]{C}. We will briefly mention the clue of the proof. One can refer to \cite[2.3]{C} for more details.

The upshot is that a genus $g$ simply branched covering curve of a standard torus can be regarded as certain \emph{lattice point} in the Hodge bundle. So the union of these covers is Zariski dense there. 

More precisely, consider $\Bbb{H}$ over $\mathcal M_{g}$ parameterizing a genus $g$ smooth curve $C$ along with a 
holomorphic 1-form $\omega$ with simple zeros, i.e., $(\omega) = q_{1}+\cdots+q_{2g-2}$. $\Bbb{H}$ can be equipped with a local coordinate chart around $(C, \omega)$ as follows. 

Take a basis $\gamma_{1},\ldots,\gamma_{4g-3}\in
H_{1}(C,q_{1},\ldots,q_{2g-2};\mathbb Z)$ the relative homology of
$C$ with $2g-2$ punctures, such that
$\gamma_{1},\ldots,\gamma_{2g}$ are the standard symplectic basis
of $H_{1}(C;\mathbb Z)$ and $\gamma_{2g+i}$ is a path connecting
$q_{1}$ and $q_{i+1}$, $i=1,\ldots, 2g-3$. The period map
$\Phi:(C,\omega)\rightarrow \mathbb C^{4g-3}$ is given by
$$ \Phi
(C,\omega)=\Big(\int_{\gamma_{1}}\omega,\ldots,\int_{\gamma_{4g-3}}\omega\Big),$$
which provides a local coordinate chart for $\Bbb H$. This was noticed by Kontsevich \cite{Ko}.  

Now take the standard lattice $\Lambda$ on $\mathbb C$ generated by $\langle1, \sqrt{-1}\rangle$, and the torus $T$ given by 
$\mathbb C/\Lambda$. Consider the coordinate $\Phi(C,\omega)=(\phi_{1},\ldots, \phi_{4g-3})\in \mathbb C^{4g-3}$. We
cite the following lemma from \cite[Lemma 3.1]{EO}.
\begin{lemma}
$\phi_{i}\in \Lambda, i=1,\ldots, 2g$ if and only if the following conditions 
hold: \\
(1) there exists a holomorphic map $f: C\rightarrow T$; \\ (2)
$\omega = f^{-1}(dz)$; \\ (3) $f$ is simply ramified at $q_{i}$, $i=1,\ldots, 2g-2$; \\ (4)
$f(q_{i+1})-f(q_{1})=\phi_{2g+i}$ mod $\Lambda$, $i=1,\ldots, 2g-3$; \\
(5) the degree of $f$ equals $\frac{\sqrt{-1}}{2} \int_{C}\omega \wedge \overline{\omega}. $
\end{lemma}

Then the corresponding locus of $\bigcup\limits_{d}^{\infty} W_{d,g}$ is Zariski dense in $\Bbb H$ and Theorem \ref{density} follows. 
 
\section{Applications}
In this section, we will apply the general theory about $W_{d,g}$ obtained previously to various examples. 

\subsection{The slope of $W_{d,2}$} \ 

In the case $g=2$, we want to work out those enumerative numbers $N_{i}$'s in Definition \ref{subsets} explicitly. The set $Cov$ in Definition \ref{Cov} is reduced to: 
$$Cov= \{ (\alpha, \beta, \gamma_{1}, \gamma_{2})\in S_{d}\times S_{d}\times S_{d} \times S_{d} \ |\ \beta^{-1}\alpha^{-1}\beta\alpha = \gamma_{1}\gamma_{2}, $$
$$\gamma_{1}, \gamma_{2}\ \mbox{are simple transpositions, and}\  \langle\alpha, \beta, \gamma_{1}, \gamma_{2}\rangle\ \mbox{acts transitively on}  \ \{1, \ldots, d\}\} . $$ 
Its subsets in Definition \ref{subsets} are reduced to: 
$$Cov_{0}=\{ (\alpha, \beta, \gamma_{1}, \gamma_{2}) \in Cov\ |\ \gamma_{1} = \gamma_{2},\  \langle\alpha, \beta\rangle\  \mbox{acts transitively on}\  \{1, \ldots, d\} \};$$
$$Cov^{(h)}_{1}= \{ (\alpha, \beta, \gamma_{1}, \gamma_{2}) \in Cov\ |\ \gamma_{1} = \gamma_{2}, \ \langle\alpha, \beta \rangle \ \mbox{acts transitively on a partition} \ (h\ |\ d-h) \};$$ 
$$Cov_{2}= \{ (\alpha, \beta, \gamma_{1}, \gamma_{2}) \in Cov\ |\  \gamma_{1} \cap \gamma_{2} = \emptyset \}; $$
$$Cov_{3}= \{ (\alpha, \beta, \gamma_{1}, \gamma_{2}) \in Cov\ |\   |\gamma_{1} \cap \gamma_{2}| = 1 \}. $$

For simplicity, we will only deal with the case when $d$ is \emph{odd}. The even case can be done similarly but with more subtle discussion. 

The enumeration below is similar to \cite[Section 3]{C}. To better state the result, introduce a notation 
$(1^{a_{1}}2^{a_{2}}\cdots d^{a_{d}})$ to denote a conjugacy class of $S_{d}$ that has $a_{i}$ cycles of length-$i$, where 
$a_{1} + 2a_{2} + \cdots da_{d} = d$. 

First, let us get rid of the equivalence relation $\sim$. We cite a simple result \cite[Lemma 3.1]{C} as follows. 

\begin{lemma} 
\label{tau}
For an element $\tau$ in $S_{d}$, if $\tau$ commutes with all the elements in a transitive subgroup of $S_{d}$, 
then $\tau$ must consist of $m$ cycles of the same length-$l$, $lm = d$. That is, $\tau$ is of type $(l^{m})$ by the above notation. 
\end{lemma}

Furthermore, define the following sums. 

\begin{definition}
For positive integers $d$ and $i$, introduce a sum 
$$\sigma_{i}(d) = \sum\limits_{l|d} l^{i}, $$ where the summation ranges over all positive factors $l$ of $d$. 
\end{definition}

Let us figure out the number $N_{0} = |Cov_{0}/\sim|$ in the first place. Consider the conjugate action of $S_{d}$ acting on the set $Cov_{0}$. Each orbit corresponds to an element of $Cov_{0}/\sim$. Introduce a new set 
$$Cov_{0}(\tau) = \{(\alpha, \beta, \gamma_{1}, \gamma_{2}) \in Cov_{0}\ |\ \tau\  \mbox{commutes with}\ \alpha, \beta, \gamma_{1}, \gamma_{2} \} $$
parameterizing data in $Cov_{0}$ that have $\tau$ as a stabilizer under the conjugate action of $S_{d}$. By Burnside's lemma, we have
 $$N_{0} = \frac{1}{d!}\sum\limits_{\tau\in S_{d}} |Cov_{0}(\tau)|. $$ 
 
 Since $\gamma_{1} = \gamma_{2}$ are the same simple transposition $(t_{1} t_{2})$, $\tau$ commutes with $\gamma_{1}, \gamma_{2}$ if and only if either $\tau$ fixes both $t_{1}, t_{2}$ or 
 $\tau$ permutes them. Now by Lemma \ref{tau} and the assumption that $d$ is odd, $\tau$ cannot have a cycle of length 2. Therefore $\tau$ must be the identity. 
 Then we simply have 
 $$ N_{0} = \frac{1}{d!} |Cov_{0}|. $$ 
  
For an element $(\alpha, \beta, \gamma_{1}, \gamma_{2})\in Cov_{0}$, $\gamma_{1} = \gamma_{2}$ can have ${d\choose 2}$ options. 
Fix one choice, then $\alpha^{-1}\beta\alpha = \beta$. In this form it is clear that if $\beta$ has two cycles of different lengths, the letters in these two cycles would never be exchanged by 
$\alpha$, which contradicts the transitivity of $\langle \alpha, \beta\rangle$. So $\beta$ must be of type $(l^{a})$, where $la = d$. 
$\beta$ has $\frac{d!}{l^{a}a!}$ choices. For a fixed $\beta$, $\alpha$ has $(a-1)! l^{a}$ choices. In total, we get 
 $$ |Cov_{0}| = {d\choose 2} \sum\limits_{l|d} \frac{d!}{l}, \ \mbox{and}\  N_{0} = \frac{d-1}{2} \sigma_{1}(d). $$

Next, let us solve $N_{1}^{(h)}, 1\leq h \leq \frac{d-1}{2}$. By the same token, we have 
$$N_{1}^{(h)} =  \frac{1}{d!} |Cov_{1}^{(h)}|. $$ 

For an element $(\alpha, \beta, \gamma_{1}, \gamma_{2})\in Cov_{1}^{(h)}$, assume that $\gamma_{1} = \gamma_{2}$ are the simple transposition permuting two letters
$h$ and $d$. Also assume that $\langle \alpha, \beta \rangle$ acts transitively on both $\{1, \ldots, h\}$ and its complement $\{h+1, \ldots, d\}$. 
The choices involved in these two steps are ${d \choose h}h(d-h)$. For a fixed choice, we have $\alpha^{-1}\beta\alpha = \beta$. 
Again due to the transitivity of $\langle \alpha, \beta \rangle$, $\beta$ must be of type $(l_{1}^{a_{1}} l_{2}^{a_{2}})$, where $a_{1}l_{1} = h, a_{2}l_{2} = d-h$. The cycles of length-$l_{1}$ act on  
$\{1, \ldots, h\}$ and those of length-$l_{2}$ act on $\{h+1, \ldots, d\}$. Such a $\beta$ has choices $\frac{i!}{l_{1}^{a_{1}}a_{1}!}\frac{(d-i)!}{l_{2}^{a_{2}}a_{2}!}. $
For a fixed $\beta$, there are $(a_{1}-1)! l_{1}^{a_{1}}(a_{2}-1)! l_{2}^{a_{2}}$ choices for $\alpha$. Multiplying the numbers all together, we get 
$$ |Cov_{1}^{(h)}| = d!\Big( \sum\limits_{a_{1}l_{1} = h, \atop a_{2}l_{2} = d-h}l_{1}l_{2}\Big), \ N_{1}^{(h)} = \sum\limits_{a_{1}l_{1} = h, \atop a_{2}l_{2} = d-h}l_{1}l_{2}. $$
In particular, 
$$N_{1} = \sum\limits_{h=1}^{\frac{d-1}{2}} N_{1}^{(h)} = \sum\limits_{h=1}^{\frac{d-1}{2}} \sum\limits_{a_{1}l_{1} = h, \atop a_{2}l_{2} = d-h}l_{1}l_{2} $$
$$ = \frac{1}{2} \sum\limits_{a_{1}l_{1}+a_{2}l_{2} = d}l_{1}l_{2}
 = \frac{1}{2}\sum\limits_{h=1}^{d-1}\sigma_{1}(h)\sigma_{1}(d-h). $$

Since the term $N_{2}$ does not appear in the slope formula, cf. Theorem \ref{slope}, we leave the enumeration of $N_{2}$ to the reader. 

Finally for $N_{3}$, we have 
 $$ N_{3} = \frac{1}{d!} |Cov_{3}|. $$ 
 
 For an element $(\alpha, \beta, \gamma_{1}, \gamma_{2})\in Cov_{3}$, $\gamma_{1}$ and $\gamma_{2}$ have one common letter. Assume that 
 $\gamma_{1} = (ac)$ and $\gamma_{2} = (ab)$, then $\alpha^{-1}\beta\alpha = \beta (abc)$. Note that this is exactly the case we analyzed in \cite[3.1]{C}, though there only the case $d$ prime was considered. 
 The key point is that $\beta$ and $\beta (abc)$ are in the same conjugacy class. $\beta$ can be of type either $(l^{a}), la =d$ or 
 $(l_{1}^{a_{1}}l_{2}^{a_{2}}), l_{1} > l_{2}, a_{1}l_{1}+a_{2}l_{2} = d.$
If $\beta$ is of type $(l^{a})$, it has $\frac{d!}{a!l^{a}}$ options. For a fixed $\beta$, there are $3a{l \choose 3}$ choices for $\gamma_{1},\gamma_{2}$, and then there are $l^{a}(a-1)!$ choices for $\alpha$. 
If $\beta$ is of type $(l_{1}^{a_{1}}l_{2}^{a_{2}}),$ there are $\frac{d!}{l_{1}^{a_{1}}a_{1}!l_{2}^{a_{2}}a_{2}!}$ choices for $\beta$. Fix a choice, and there are 
$3a_{1}l_{1}a_{2}l_{2}$ choices for $\gamma_{1}, \gamma_{2}$. For fixed $\beta, \gamma_{1}, \gamma_{2}$, there are $l_{1}^{a_{1}}(a_{1}-1)!l_{2}^{a_{2}}(a_{2}-1)!$ choices for $\alpha$. 
Eventually we get 
$$ N_{3} = \frac{1}{2}\sum\limits_{l|d}l(l-1)(l-2) + 3\sum\limits_{a_{1}l_{1}+a_{2}l_{2} = d, \atop l_{1} > l_{2}}l_{1}l_{2} $$
$$ = \frac{3}{2}\Big(\sum\limits_{h=1}^{d-1}\sigma_{1}(h)\sigma_{1}(d-h)\Big)- \big(\frac{3}{2}d - 1\big)\sigma_{1}(d) + \frac{1}{2}\sigma_{3}(d). $$ 

\begin{corollary}
\label{relation} 
When $g=2$, the following equality holds for $N_{0}, N_{1}$ and $N_{3}$: 
$$ 5N_{3} = 27N_{1} - 9N_{0} .$$ 
\end{corollary}

\begin{proof}
We provide two proofs here. The first one is direct and only for odd $d$. Plugging in the above expressions of those $N_{i}$'s, the desired equality is equivalent to 
$$\sum_{h=1}^{d-1}\sigma_{1}(h)\sigma_{1}(d-h)=\big(\frac{1}{12}-\frac{d}{2}\big)\sigma_{1}(d)+\frac{5}{12}\sigma_{3}(d),$$
which has been proved in \cite[Lemma 3.4]{C}. 

The other proof is indirect but more interesting. The moduli space $\overline{\mathcal M}_{2}$ is special in that the relation 
$\lambda = \frac{\delta_{0}}{10} + \frac{\delta_{1}}{5}$ holds. Note that in Proposition \ref{Delta} and \ref{Lambda}, we have already known the intersection numbers 
$$W\ldotp \lambda = \frac{N_{0}+N_{1}}{4} + \frac{N_{3}}{36}, $$
$$W\ldotp \delta_{0} = 2N_{0}, \ W\ldotp \delta_{1} = 2 N_{1}.$$
Plugging them into $\lambda = \frac{\delta_{0}}{10} + \frac{\delta_{1}}{5}$ will give us exactly what we want. 
\end{proof}

\begin{corollary}
\label{odd}
For $d$ odd and $g=2$, $\lim\limits_{d\to\infty} s(W_{d,2}) = 5$. 
\end{corollary}

\begin{proof} Using Corollary \ref{relation}, we have 
$$s(W_{d,2}) = \frac{10(N_{0}+N_{1})}{N_{0}+2N_{1}}$$ 
which goes to 5 for large $d$, since 
asymptotically $N_{0}\sim O(d^{2})$ and $N_{1} \sim O(d^{3})$. 
\end{proof}

This provides a proof of Corollary \ref{g2} for the case $d$ odd. We leave the other case $d$ even for the reader. 

\begin{remark}
We already mentioned in the introduction section that the one-parameter family $W$ may not be as good as $Y$ from the viewpoint of bounding slopes. 
For $g=2$, a curve $B$ in $\overline{\mathcal M}_{2}$ has the sharp bound 10 as its slope if and only if $B$ does not meet $\Delta_{1}$. 
$Y$ does not meet $\Delta_{i}$ for any $i > 0$, cf. \cite[Remark 3.6]{C}. But $W$ meets $\Delta_{1}$ since $N_{1}$ is non-zero. It seems 
that the most degenerate case $Y$ probably should give us the best bound for $s_{g}$ among all one-parameter families of covers of elliptic curves.  
This is also one of the reasons why we do not pursue a further estimate for $s(W)$ in general. 
\end{remark}

\subsection{The slope of $W_{2,g}$}\ 

When $d=2$, let us consider double covers of an elliptic curve. This case is easy to analyze since we are in $S_{2}$.  
For $(\alpha, \beta, \gamma_{1}, \ldots, \gamma_{2g-2})\in Cov$, all $\gamma_{i}$'s are the same simple transposition $(12)$ permuting the two sheets of the cover. 
Then $\alpha, \beta$ can take any elements in $S_{2}$. There are in total four different combinations and all of them are not equivalent, so $N=4$. In particular, $N_{0} = 3, N_{1} = 1, N_{2} = N_{3} = 0$. 
By the slope formula, we get $s(W_{2,g}) = 8$ a constant. This example was also noticed by Xiao long time ago, cf. \cite[Ex. 3]{X}. 

Moreover, in this case the four equivalence classes in $Cov/\sim$ can be sent to each other by the monodromy actions in Theorem \ref{monodromy}. Hence, $W_{2,g}(E)$ is an irreducible 
degree 4 covers of $E$. If an effective divisor $D$ on \Mg has slope lower than 8, $D\ldotp W_{2,g}(E) < 0$ and $D$ must contain the image of $W_{2,g}(E)$ in $\overline{\mathcal M}_{g}$. 
Therefore, we have the following conclusion. 

\begin{corollary}
The base locus of an effective divisor $D$ with slope lower than 8 on \Mg contains the locus of genus $g$ curves that admit a double cover to an elliptic curve. 
\end{corollary}

\subsection{The genus of $W_{d,g}$}\ 

We can easily obtain the genus of $W_{d,g}$ by Proposition \ref{ram}. Since $W_{d,g}\rightarrow X$ is degree $N$ finite map triply ramified at 
a point corresponding to a local degeneration of covers in $Cov_{3}$, Riemann-Hurwitz says that 
$$2g(W_{d,g}) - 2 =  N(2g(X)-2) + 2k N_{3}. $$   Hence, we have the genus formula for $W_{d,g}$. 

\begin{proposition} The genus of $W_{d,g}$ equals 
\label{genusw}
$$g(W_{d,g}) = 1 + N(g(X) -1 ) + k N_{3}. $$ 
\end{proposition}

For the case $g=2$, $d$ odd and $X = E$, we have a more explicit expression using the enumeration for $N_{3}$ obtained in section 3.1.
\begin{corollary}
\label{genus2} 
When $g=2$, the genus of $W_{d,2}(E)$ equals 
$$ g(W_{d,2}(E)) = 1 + \frac{8}{9} \big( \sigma_{3}(d) - 2d\sigma_{1}(d) + \sigma_{1}(d)\big),$$
where $\sigma_{i}(d) = \sum\limits_{l|d} l^{i}. $
\end{corollary}

For the case $d=2$, we also have an expression for the genus of $W_{2,g}$ by using Proposition \ref{genusw} and results from the previous subsection. 
\begin{corollary}
When $d=2$, the genus of $W_{2,g}$ equals 
$$g(W_{2,g}) = 4g(X) - 3.$$ 
\end{corollary}

\begin{remark}
Notice that if $W$ is reducible, each component of $W$ also admits a finite map to $X$. As long as we know the subset of $Cov/\sim$ corresponding to a component of $W$, the above 
formulae can be easily modified to work out the genus of that component. 
\end{remark}

\subsection{The components of $W_{d,g}$}\ 

We mentioned before that the monodromy criterion in Theorem \ref{monodromy} may not be easy to apply in practice. However, for small genus or degree, 
we are able to use it to derive some information about the components of $W_{d,g}$. Here we introduce another notation $(a_{1}a_{2}\cdots a_{n})$ to denote 
a cycle of a permutation in $S_{d}$ that sends $a_{i}$ to $a_{i+1}, n\leq d$. 

Firstly, let us focus on a concrete example when $g=2$ and $d=3$. 

\begin{example}
\label{d3g2}
For the case $W_{3,2}$, genus 2 triple covers of an elliptic curve, let $S_{3}$ act on the three letters $\{1, 2, 3 \}$. Then
$Cov/\sim$ consists of the following representatives $(\alpha, \beta, \gamma_{1}, \gamma_{2}).$

$(1)\  \alpha = id, \beta = (123), \gamma_{1} = (12), \gamma_{2} = (12); \ (2)\ \alpha = (123), \beta = id, \gamma_{1} = (12), \gamma_{2} = (12); $

$ (3)\  \alpha = (123), \beta = (123), \gamma_{1} = (12), \gamma_{2} = (12); \ (4)\  \alpha = (123), \beta = (132), \gamma_{1} = (12), \gamma_{2} = (12). $

$ (5) \ \alpha = id, \beta = (13), \gamma_{1} = (12), \gamma_{2} = (12); \ (6) \ \alpha = (13), \beta = id, \gamma_{1} = (12), \gamma_{2} = (12); $

$ (7) \ \alpha = (13), \beta = (13), \gamma_{1} = (12), \gamma_{2} = (12); \ (8) \ \alpha = (23), \beta = (12), \gamma_{1} = (12), \gamma_{2} = (13); $

$ (9) \ \alpha = (123), \beta = (12), \gamma_{1} = (12), \gamma_{2} = (13); \ (10)\  \alpha = (12), \beta = (13), \gamma_{1} = (12), \gamma_{2} = (13); $

$ (11)\  \alpha = (123), \beta = (13), \gamma_{1} = (12), \gamma_{2} = (13); \ (12)\  \alpha = (13), \beta = (23), \gamma_{1} = (12), \gamma_{2} = (13); $

$ (13) \ \alpha = (123), \beta = (23), \gamma_{1} = (12), \gamma_{2} = (13); \ (14) \ \alpha = (12), \beta = (132), \gamma_{1} = (12), \gamma_{2} = (13); $

$ (15) \ \alpha = (13), \beta = (132), \gamma_{1} = (12), \gamma_{2} = (13); \ (16) \ \alpha = (23), \beta = (132), \gamma_{1} = (12), \gamma_{2} = (13). $

Hence, $N = 16$. Moreover, $Cov_{0}/\sim$ consists of cases (1)-(4), $Cov_{1}/\sim$ consists of cases (5)-(7), $Cov_{3}/\sim$ consists of cases (8)-(16) and 
$Cov_{2} = \emptyset$. Therefore, we get $N_{0} = 4, N_{1} = 3, N_{2} = 0, N_{3} = 9$. Note that they do satisfy the relation in Corollary \ref{relation}. 
The slope of $W_{3,2}$ equals 7 by the slope formula. Finally, the monodromy actions in Theorem \ref{monodromy} act transitively on the 16 equivalence classes of $Cov/\sim$, 
so $W_{3,2}(E)$ is actually irreducible as a 16-sheeted cover of $E$.   
The number of components of $W$ bounds from above the number of components of the Hurwitz space $\mathcal H_{d,g}$, since every component of $\mathcal H_{d,g}$ is an unramified cover of $\mathcal{M}_{1,2g-2}$. 
Hence, $\mathcal H_{3,2}$ is also irreducible. 
\end{example}

Next, we will present an example when $W$ has more than one irreducible components. 

\begin{example}
\label{d4g2}
Consider the case when $g=2$ and $d=4$. 
Let a cover $\pi_{1}$: $C_{1}\rightarrow E$ correspond to an element 
$\alpha = \beta = (13)(24), \gamma_{1} = \gamma_{2} = (12)$ in $Cov$. The group generated by $\alpha, \beta, \gamma_{1}, \gamma_{2}$ is a proper subgroup of $S_{4}$. 
On the other hand, let $\pi_{2}$: $C_{2}\rightarrow E$ correspond to another element $\alpha = \beta = (1234), \gamma_{1} = \gamma_{2} = (12)$ in $Cov$. This time 
$\alpha, \beta, \gamma_{1}, \gamma_{2}$ generate the whole group $S_{4}$. Since the conjugacy type of the subgroup generated by $\alpha, \beta, \gamma_{1}, \gamma_{2}$ 
is invariant under the monodromy actions, the two covers $\pi_{1}$ and $\pi_{2}$ are necessarily contained in different components of $W_{4,2}(E)$. Hence, $W_{4,2}(E)$ is reducible. 
\end{example}

The conjugacy type of the subgroup generated by $\alpha, \beta, \gamma_{1}, \ldots, \gamma_{2g-2}$ in $S_{d}$ should be considered as a \emph{parity} to distinguish connected components of 
$\mathcal H_{d,g}$. Since $\mathcal H_{d,g}$ is a finite unramified cover of $\mathcal{M}_{1,2g-2}$, it does not make a difference to consider connected or irreducible components of $\mathcal H_{d,g}$. 
To the author's best knowledge, there are almost no results in general for the number of components of $\mathcal H_{d,g}$ when $g$ is large. As for the beginning case $g=2$, Kani studied 
the Hurwitz space $\mathcal H_{d,2}$ over a field $K$ of characteristic prime to $2d$. One of the related results is the following, cf. \cite[Cor. 1.3]{Ka} 

\begin{proposition}
If $d$ is odd, then $\mathcal H_{d,2}$ has $\sum\limits_{n|d, \ n<d}\sigma_{1}(n)$ irreducible components.  
\end{proposition}

Thus $\mathcal H_{d,2}$ is irreducible if and only if $d$ is prime. It coincides with our analysis in Example \ref{d3g2} and \ref{d4g2}. We also post the following interesting question for a further study. 

\begin{question}
How many irreducible components does $\mathcal{H}_{d,g}$ have? In particular, does the conjugacy type of $\langle \alpha, \beta, \gamma_{1}, \ldots, \gamma_{2g-2}\rangle$ uniquely 
determine a component of $\mathcal{H}_{d,g}$? 
\end{question}

Department of Mathematics, Harvard University, 1 Oxford Street, Cambridge, MA 02138 \par
{\it Email address:} dchen@math.harvard.edu

\end{document}